\documentclass[11pt]{article}
\usepackage{amsmath,amssymb}

\newtheorem{propo}{{\bf Proposition}}[section]
\newtheorem{coro}[propo]{{\bf Corollary}}
\newtheorem{lemma}[propo]{{\bf Lemma}} \newtheorem{theor}[propo]{{\bf
Theorem}} \newtheorem{ex}{{\sc Example}}[section]

\newenvironment{proof}{{\bf Proof.}}{$\Box$}

\def\N{{\mathbb N}}

\begin{document}

\vspace*{1.0in}

\begin{center} LATTICE ISOMORPHISMS OF LEIBNIZ ALGEBRAS
\end{center}
\bigskip

\begin{center} DAVID A. TOWERS 
\end{center}
\bigskip

\begin{center} Department of Mathematics and Statistics

Lancaster University

Lancaster LA1 4YF

England

d.towers@lancaster.ac.uk 
\end{center}
\bigskip

\begin{abstract} Leibniz algebras are a non-anticommutative version of Lie algebras. They play an important role in different areas of mathematics and physics and have attracted much attention over the last thirty years. In this paper we investigate whether conditions such as being a Lie algebra, cyclic, simple, semisimple, solvable, supersolvable or nilpotent in such an algebra are preserved by lattice isomorphisms.
\end{abstract}

\section{Introduction}
\medskip

An algebra $L$ over a field $F$ is called a {\em Leibniz algebra} if, for every $x,y,z \in L$, we have
\[  [x,[y,z]]=[[x,y],z]-[[x,z],y]
\]
In other words the right multiplication operator $R_x : L \rightarrow L : y\mapsto [y,x]$ is a derivation of $L$. As a result such algebras are sometimes called {\it right} Leibniz algebras, and there is a corresponding notion of {\it left} Leibniz algebras, which satisfy
\[  [x,[y,z]]=[[x,y],z]+[y,[x,z]].
\]
Clearly the opposite of a right (left) Leibniz algebra is a left (right) Leibniz algebra, so, in most situations, it does not matter which definition we use. Leibniz algebras which satisfy both the right and left identities are sometimes called {\em symmetric} Leibniz algebras.
\par
 
Every Lie algebra is a Leibniz algebra and every Leibniz algebra satisfying $[x,x]=0$ for every element is a Lie algebra. They were introduced in 1965 by Bloh (\cite{bloh}) who called them $D$-algebras, though they attracted more widespread interest, and acquired their current name, through work by Loday and Pirashvili ({\cite{loday1}, \cite{loday2}). They have natural connections to a variety of areas, including algebraic $K$-theory, classical algebraic topology, differential geometry, homological algebra, loop spaces, noncommutative geometry and physics. A number of structural results have been obtained as analogues of corresponding results in Lie algebras.
\par

The {\it Leibniz kernel} is the set $I=$ span$\{x^2:x\in L\}$. Then $I$ is the smallest ideal of $L$ such that $L/I$ is a Lie algebra.
Also $[L,I]=0$.
\par

We define the following series:
\[ L^1=L,L^{k+1}=[L^k,L]  (k\geq 1) \hbox{ and } L^{(0)}=L,L^{(k+1)}=[L^{(k)},L^{(k)}] (k\geq 0).
\]
Then $L$ is {\em nilpotent of class n} (resp. {\em solvable of derived length n}) if $L^{n+1}=0$ but $L^n\neq 0$ (resp.$ L^{(n)}=0$ but $L^{(n-1)}\neq 0$) for some $n \in \N$. It is straightforward to check that $L$ is nilpotent of class n precisely when every product of $n+1$ elements of $L$ is zero, but some product of $n$ elements is non-zero.The {\em nilradical}, $N(L)$, (resp. {\em radical}, $R(L)$) is the largest nilpotent (resp. solvable) ideal of $L$.
\par

The set of subalgebras of a nonassociative algebra forms a lattice under the operations of union, $\cup$, where the union of two subalgebras is the subalgebra generated by their set-theoretic union, and the usual intersection, $\cap$. The relationship between the structure of a Lie algebra $L$ and that
of the lattice ${\cal L}(L)$ of all subalgebras of $L$ has been studied by many authors. Much is known about modular subalgebras
(modular elements in ${\cal L}(L)$) through a number of investigations including \cite{1,2,3,4,5,6}. Other lattice conditions, together with
their duals, have also been studied. These include semimodular, upper semimodular, lower semimodular, upper modular, lower modular
and their respective duals (see \cite{7} for definitions). For a selection of results on these conditions see \cite{14,8,10,kol,11,sch,13,15,9,12}.
\par

The subalgebra lattice of a Leibniz algebra, however, is rather different; in a Lie algebra every element generates a one-dimensional subalgebra, whereas in a Leibniz algebra elements can generate subalgebras of any dimension. So, one could expect different results to hold for Leibniz algebras anf this has been shown to be the case in \cite{st}..
\par

Of particular interest is the extent to which important classes of Leibniz algebras are determined by their subalgebra lattices. In order to investigate this question we introduce the notion of a lattice isomorphism. If we denote the subalgebra lattice of $L$ by ${\mathcal L}(L)$, then a {\em lattice isomorphism} from $L$ to $L^*$ is a bijective map $\theta : {\mathcal L}(L) \rightarrow {\mathcal L}(L^*)$ such that $\theta(A\cup B)=\theta(A)\cup \theta(B)$ and $\theta(A\cap B)=\theta(A)\cap \theta(B)$ for all $A,B\in {\mathcal L}(L)$. If $L$ is a Lie algebra over a field of characteristic zero the following were proved in \cite{latt}.

\begin{theor}\label{t:one}
\begin{itemize}
\item[(i)] If $L$ is simple then either
\begin{itemize}
\item[(a)] $L^*$ is simple, or
\item[(b)] $L$ is three-dimensional non-split simple and $L^*$ is two-dimensional.
\end{itemize}
\item[(ii)] If $L$ is semisimple then either
\begin{itemize}
\item[(a)] $L^*$ is semisimple, or
\item[(b)] $L$ is three-dimensional non-split simple and $L^*$ is two-dimensional.
\end{itemize}
\item[(iii)] If $\dim L,L^*>2$ and $R$ is the radical of $L$, then $R^*$ is the radical of $L^*$.
\item[(iv)] If $L$ is supersolvable of dimension $>2$, then $L^*$ is supersolvable.
\end{itemize}
\end{theor}

In \cite{10} the following was proved.

\begin{theor} If $L$ is a solvable Lie algebra over a perfect field of characteristic different from $2,3$, then either
\begin{itemize}
\item[(i)] $L^*$ is solvable, or
\item[(ii)] $L^*$ is three-dimensional non-split simple.
\end{itemize}
\end{theor}

We say that a Lie algebra $L$ is {\em almost abelian} if it is a split extension $L=L^2\dot{+} Fa$ with ad\,$a$ acting as the identity map on the abelian ideal $L^2$; $L$ is quasi-abelian if it is abelian or almost abelian. The quasi-abelian Lie algebras are precisely the ones in which every subspace is a subalgebra. The following is well-known and easy to show. 

\begin{propo} If $L$ is a quasi-abelian Lie algebra over a field of characteristic zero then $L^*$ is quasi-abelian unless $\dim L=2$ and $L^*$ is three-dimensional non-split simple.
\end{propo}

In this paper we consider corresponding results for Leibniz algebras. First, in section two, we show that cyclic Leibniz algebras are characterised by their subalgebra lattice, and that a non-Lie Leibniz algebra cannot be lattice isomorphic to a Lie algebra. In section three we see that if $L$ is a non-Lie simple or semisimple Leibniz algebra then so is $L^*$. In section four, it is shown that if $L$ is a non-Lie solvable or supersolvable Leibniz algebra then so is $L^*$. It is also proved that the radical of a non-Lie Leibniz algebra is preserved by lattice isomorphisms. The final section is devoted to showing that if $L$ is a non-Lie nilpotent Leibniz algebra then so is $L^*$. Most of the above results are over fields of characteristic zero.
\par

Throughout, $L$ will denote a finite-dimensional Leibniz algebra over a field $F$. Algebra direct sums will be denoted by $\oplus$, whereas vector space direct sums will be denoted by $\dot{+}$. The notation `$A\subseteq B$' will indicate that $A$ is a subset of $B$, whereas `$A\subset B$' will mean that $A$ is a proper subset of $B$. If $A$ and $B$ are subalgebras of $L$ we will write $\langle A,B \rangle$ for $A\cup B$.
\par

The {\it centre} of $L$ is $Z(L)=\{z\in L \mid [z,x]=[x,z]=0$ for all $x\in L\}$. The Frattini ideal of $L$, $\phi(L)$, is the largest ideal of $L$ contained in all maximal subalgebras of $L$
.
\section{Cyclic Leibniz algebras}
The only previous paper that we are aware of on this topic is by Barnes (\cite{barnes2}). The following example shows that the Leibniz kernel of a non-Lie Leibniz algebra is not necessarily preserved by a lattice isomorphism.

\begin{ex}\label{e:ex1}
Let $L=Fb+Fa$ where the only non-zero products are $[b,b]=a$, $[a,b]=a$. Then the only subalgebras of $L$ are ${0}$, $Fa$, $F(b-a)$ and $L$, and $I=Fa$. Then we can define a lattice automorphism of $L$ which interchanges $Fa$ and $F(b-a)$, and the latter is not an ideal of $L$ as $[b,b-a]=a$. 
\end{ex}

Barnes called the above example the {\em diamond} algebra because of the structure of its lattice of subalgebras as a Hasse diagram, but that name has since been used for a different Leibniz algebra. He further showed that this example is exceptional in the following result.

\begin{theor}\label{t:kernel} (\cite[Theorem 3.1]{barnes2}) Let $L, L^*$ be Leibniz algebras with Leibniz kernels $I, I^*$ respectively, and let $\theta : L \rightarrow L^*$ be a lattice isomorphism. Suppose that $\dim L \geq 3$. Then $\theta(I)=I^*$.
\end{theor}

However, this paper does not appear to have been followed by further investigations into the subalgebra structure of a Leibniz algebra. Theorem \ref{t:kernel}, of course, has an immediate corollary. 

\begin{coro}\label{t:lie} Let $L$ be a non-Lie Leibniz algebra. Then $L$ cannot be lattice isomorphic to a Lie algebra $L^*$. 
\end{coro}
\begin{proof} If $\dim L\geq 3$ then $I\neq 0$ if and only if $I^*\neq 0$. If $\dim L=2$ there are only two possibilities for $L$, both of them cyclic with basis $x, x^2$. In the first, $[x^2,x]=0$ and the only proper subalgebra is $Fx^2$, and in the second, $[x^2,x]=x^2$ and the only proper subalgebras are $Fx^2$ and $F(x-x^2)$. However, every Lie algebra of dimension greater than one has more than two proper subalgebras. 
\par

There is no non-Lie Leibniz algebra of dimension one.
\end{proof}
\medskip

A Leibniz algebra $L$ is called {\em cyclic} if it is generated by a single element. In this case, $L$ has a basis $x, x^2, \ldots, x^n (n>1)$ and products $[x^i,x]=x^{i+1}$ for $1\leq i\leq n-1$, $[x^n,x]=\alpha_2 x^2 + \ldots + \alpha_n x^n$, all other products being zero. Then we have the following.

\begin{theor}\label{t:cyclic} If $L$ is a cyclic Leibniz algebra over an infinite field $F$, then $L^*$ is also a cyclic Leibniz algebra of the same dimension.
\end{theor}
\begin{proof} Over an infinite field a Leibniz algebra is cyclic if and only if it has finitely many maximal subalgebras, by \cite[Corollary 2.3]{st}. Moreover, the length of a maximal chain of subalgebras of a cyclic algebra is equal to its dimension. 
\end{proof}

\begin{coro}\label{c:cyclic} If $L$ is a nilpotent cyclic Leibniz algebra, then $L^* \cong L$.
\end{coro}
\begin{proof} A nilpotent cyclic Leibniz algebra has only one maximal subalgebra, namely $I$, its Leibniz kernel. It follows that $L^*$ is nilpotent of the same dimension. Note that the restriction on the field is unnecessary here, since, if $M$ is the only maximal subalgebra of $L$ and $x\in L\setminus M$, we must have $L=\langle x\rangle$.
\end{proof}
\medskip

{\bf Note} that, in both of the above results, if $L=\langle x\rangle$ then $L^*=\langle x^*\rangle$, since $x$ does not belong to any of the maximal subalgebras of $L$, and this is inherited by $x^*$ in $L^*$.

\begin{propo}\label{p:cyclic} Let $L=A\dot{+}Fx$ be a non-Lie Leibniz algebra in which $A$ is a minimal abelian ideal of $L$ and $x^2=0$. Then $L$ is cyclic and $A=I$.
\end{propo}
\begin{proof} Since $L$ is not a Lie algebra, $A=I$, $[L,A]=0$ and $[A,L]\neq 0$, so $[A,x]=A$. Let $0\neq a\in A$. Then $(x+a)^n=R_x^{n-1}(a)$ for $n\geq 2$, which implies that $[(x+a)^n,x]=R_x^n(a)\in \langle x+a\rangle$ for $n\geq 1$. Hence $\langle x+a\rangle\cap A$ is an ideal of $L$ and so equals $A$ or $0$. However, the latter implies that $[a,x]=0$, whence $A=Fa$ and $[A,x]=0$, a contradiction. It follows that $L=\langle x+a\rangle$.
\end{proof}

\section{Semisimple Leibniz algebras}
The following useful result was proved by Barnes in \cite{barnes1}. Note that we have modified the statement to take account of the fact that Barnes' result is stated for left Leibniz algebras and we are dealing with right Leibniz algebras.

\begin{lemma}\label{l:min} Let $A$ be a minimal ideal of the Leibniz algebra $L$. Then $[L,A]=0$ or $[x,a]=-[a,x]$ for all $a\in A$, $x\in L$.
\end{lemma}
\medskip

A Leibniz algebra $L$ is called {\em simple} if its only ideals of $L$ are $0$, $I$ and $L$, and $L^2\neq I$. If $L/I$ is a simple Lie algebra then $L$ is not necessarily a simple Leibniz algebra. It is said to be {\em semisimple} if $R(L)=I$. This definition agrees with that of a semisimple Lie algebra, since, in this case, $I=0$. Semisimple Leibniz algebras are not necessarily direct sums of simple Leibniz algebras (see \cite{dms} or \cite{akoz}).
\par

We have the following version of Levi's Theorem.

\begin{theor}\label{t:levi}(Barnes \cite{barnes}) Let $L$ be a finite-dimensional Leibniz algebra over a field of characteristic $0$. Then there is a semisimple Lie subalgebra $S$ of $L$ such that $L=S\dot{+}R(L)$.
\end{theor}

We shall need the following result which was proved by Gein in \cite[p. 23]{gein}.

\begin{lemma}\label{l:gein} Let $S$ be a three-dimensional non-split simple Lie algebra, and let $R$ be an irreducible $S$-module. Then, for any $s\in S$, $R$ has an ad\,$s$-invariant subspace of dimension less than or equal to two.
\end{lemma}

If $U$ is a subalgebra of $L$ and $0=U_0 < U_1 < \ldots < U_n=U$ is a maximal chain of subalgebras of $U$ we will say that $U$ has {\em length $n$}. 

\begin{theor}\label{t:simple} Let $L=S\dot{+} A$ be a Leibniz algebra over a field of characteristic zero, where $S$ is a three-dimensional non-split simple Lie algebra and $A$ is a minimal abelian ideal of $L$. Then $L^*$ has a simple Lie subalgebra.
\end{theor}
\begin{proof} Suppose that $L^*$ does not have a simple Lie subalgebra. Then $R(L^*)\neq 0$, by Theorem \ref{t:levi}, and so $L^*$ has a minimal abelian ideal $B^*$. As $S^*$ is a maximal subalgebra of $L^*$ we must have that $L^*=S^*\dot{+}B^*$. If $\dim A=1$ we have that $L=S\oplus A$ is a Lie algebra and hence, so is $L^*$, giving that $S^*\cong S$, by \cite[Lemma 3.3]{latt} and contradicting our supposition. Hence $\dim A\geq 2$.
\par

Now maximal subalgebras of $L$ are of two types: they are isomorphic to $S$, and so have length $2$, or they are of the form $Fs\dot{+} A$, where $s\in S$, and so are solvable of length at least $3$. Moreover, $A$ is the intersection of those of the second type. The same must be true of the maximal subalgebras of $L^*$ and so $B^*=A^*$ and $L^*=S^*\dot{+} A^*$. Also, $\dim S^*=2$, by Theorem \ref{t:one}. Now $\phi(L^*)=(\phi(L))^*=0$, so $L^*=A^*\dot{+}C^*$, where $C^*$ is abelian, by \cite[Corollary 2.9]{frat}. Since $S^*\cong C^*$, we have that $S^*$ is abelian.
\par

Let $0\neq s^*\in S^*$, $0\neq a^*\in A^*$ and let $f(\theta)$ be the polynomial of smallest degree for which $f(R_{s^*})(a^*)=0$. It follows from the fact that $S^*$ is abelian that $\{x^*\in A^* : f(R_{s^*})(x^*)=0\}$ is an ideal of $L^*$, and hence that it coincides with $A^*$. Clearly then $f(\theta)$ is the minimum polynomial of $R_{s^*}|_{A^*}$.
\par
Suppose that there is an $s_1^*\in S^*$ for which the minimum polynomial for $R_{s_1^*}$ has degree two. and let this polynomial be $f(\theta)=\theta^2-\lambda_2\theta-\lambda_1$. Pick $s_2^*\in S^*$ linearly independent of $s_1^*$. Then 
\begin{align*}
[[a^*,s_1^*],s_1^*] & =\lambda_1a^*+\lambda_2[a^*,s_1^*] \hbox{ and } \\
[[a^*,s_2^*],s_2^*] & = \alpha_1a^*+\alpha_2[a^*,s_2^*] \hbox{ so }\\
[[a^*,s_1^*],s_2^*] & = [a^*,[s_1^*,s_2^*]]+[[a^*,s_2^*],s_1^*] = [[a^*,s_2^*],s_1^*] \\
 & = \beta_1a^*+\beta_2[a^*,s_1^*]+\beta_3[a^*,s_2^*],
\end{align*} since $[[a^*,s_1^*+s_2^*],s_1^*+s_2^*] \in Fa^*+F[a^*,s_1^*+s_2^*]$. Now
\[
[[[a^*,s_2^*],s_1^*],s_1^*]=\lambda_1[a^*,s_2^*]+\lambda[[a^*,s_2^*],s_1^*], 
\] so 
\begin{align*}
(\beta_2\lambda_1+\beta_3\beta_1)a^*+(\beta_1+\beta_2\lambda_2)[a^*,s_1^*]+\beta_3^2[a^*,s_2^*] \\
=\lambda_2\beta_1a^*+\lambda_2\beta_2[a^*,s_1^*]+(\lambda_1+\lambda_2\beta_3)[a^*,s_2^*].
\end{align*}
Since $f(\theta)$ is irreducible, $\beta_3^2\neq \lambda_1+\lambda_2\beta_3$ and so $[a^*,s_2^*]=\gamma_1a^*+\gamma_2[a^*,s_1^*]$. Hence $A^*$ is two dimensional.
\par

Put $A=Fa+F[a,s]$. Choose $s_1, s_2$ to be elements of $S$ such that $s, s_1, s_2$ are linearly independent. Then $[a,s_1]= \alpha a+\beta [a,s]$ and $[a,s_2]=\gamma a+\delta [a,s]$ for some $\alpha, \beta, \gamma, \delta \in F$. Thus $[a,s_1-\beta s]=\alpha a$ and $[a,s_2-\delta s]=\gamma a$. But $s_1-\beta s$ and $s_2-\delta s$ are linearly independent, so $$[a,S]=[a,<s_1-\beta s, s_2-\delta s>]\subseteq Fa$$ and $A$ is one dimensional, a contradiction.
\end{proof}

\begin{coro}\label{c:ss} Let $L$ be a non-Lie semisimple Leibniz algebra over a field of characteristic zero. Then $L^*$ is a non-Lie semisimple Leibniz algebra.
\end{coro}
\begin{proof} We have that $I\neq 0$, so $L=I\dot{+} S$ where $S$ is a semisimple Lie algebra, by Theorem \ref{t:levi}. Then $L^*/I^*$ is a semisimple Lie algebra or $\dim L^*/I^* =2$ and $S$ is $3$-dimensional non-split simple, by Theorem \ref{t:one}(ii). 
\par

Suppose that the latter holds, so $L^*$ is solvable. Let $A$ be a minimal ideal of $L$ inside $I$ and put $B=A\dot{+} S$. Then $B^*$ has a simple Lie subalgebra, by Theorem \ref{t:simple} and $L^*$ cannot be solvable. Hence the former holds and  $L^*$ is a non-Lie semsimple Leibniz algebra.
\end{proof}
\medskip

A subalgebra $U$ of $L$ is called {\em upper semi-modular} if $U$ is a maximal subalgebra of $\langle U,B\rangle$ for every subalgebra $B$ of $L$ such that $U\cap B$ is maximal in $B$. Using this concept we have a further corollary.

\begin{coro}  Let $L$ be a non-Lie simple Leibniz algebra over a field of characteristic zero. Then $L^*$ is a non-Lie simple Leibniz algebra.
\end{coro}
\begin{proof} We have that $L=I\dot{+} S$ where $S$ is a simple Lie subalgebra of $L$ and $I\neq 0$. If $L^*/I^*$ is not simple then $S$ must be three-dimensional non split simple, by Theorem \ref{t:one}(i), and we get a contradiction as in the previous corollary.
\par

Let $0\neq A^*$ be an ideal of $L^*$. Suppose first that $A^*\subseteq I^*$. Then $A$ is an upper semi-modular subalgebra of $L$ with $A\subseteq I$. Let $s\in S$. Then $A\cap Fs=0$ is a maximal subalgebra of $Fs$. Hence $A$ is a maximal subalgebra of $C=\langle A,s \rangle$. Now $A\subseteq C\cap I\subset C$, so $A=C\cap I$. Thus $[s,A], [A,s]\subseteq C\cap I=A$, so $A$ is an ideal of $L$, whence $A=I$. It follows that $A^*=I^*$.
\par

Next, suppose that $A^*\not \subseteq I^*$. Then $I^*+A^*=L^*$ and $I^*\cap A^*=I^*$ or $0$, by the previous paragraph. The former implies that $A^*=L^*$; the latter gives that $L^*=I^*\oplus A^*$ giving $I^*=0$ and $L^*=A^*$ again. 
\par

Clearly $(L^*)^2\neq I^*$, so $L^*$ is a non-Lie simple Leibniz algebra.
\end{proof}

\section{Solvable and supersolvable Leibniz algebras}
\begin{propo}\label{p:solv} Let $L$ be a non-Lie solvable Leibniz algebra over a field of characteristic zero. Then $L^*$ is a non-Lie solvable Leibniz algebra.
\end{propo}
\begin{proof} Let $L$ be a minimal counter-example. Then $L^*$ has a semisimple Lie subalgebra $S^*$, and so $S (\neq L)$ must be two dimensional and $S^*$ must be three-dimensional non-split simple. Moreover, $L^*=S^*\dot{+} A^*$, where $A^*$ is a minimal ideal of $L^*$, since, otherwise, this is a smaller counter-example. But then $L$ has a simple subalgebra, by Theorem \ref{t:simple}, a contradiction.
\end{proof}

\begin{lemma}\label{l:int} Let $L$ be a Leibniz algebra over a field of characteristic zero. Then the radical, $R$, of $L$ is the intersection of the maximal solvable subalgebras of $L$.
\end{lemma}
\begin{proof} Let $\Gamma$ be the intersection of the maximal solvable subalgebras of $L$. Then $R\subseteq \Gamma$. Furthermore, $\Gamma$ is invariant under all automorphisms of $L$, and hence is invariant under all derivations of $L$, by \cite[Corollary 3.2]{liefrat}. It follows that $\Gamma$ is a right ideal of $L$. But
$[x,y]+[y,x]\in I\subseteq \Gamma$ for all $x\in L$, $y\in \Gamma$, so $\Gamma$ is an ideal of $L$, whence $\Gamma\subseteq R$.
\end{proof}
\medskip

Then we have the following corollaries to Proposition \ref{p:solv}.

\begin{coro} Let $L$ be a non-Lie Leibniz algebra over a field of characteristic zero, and let $R$ be the radical of $L$. Then $R^*$ is the radical of $L^*$.
\end{coro}
\begin{proof} Let $U$ be a maximal solvable subalgebra of $L$. If $U$ is non-Lie then $U$ is solvable, by Proposition \ref{p:solv}. If $U$ is Lie, then $U^*$ is solvable, unless $\dim U=2$ and $U^*$ is three-dimensional non-split simple. If $R^*=0$ then $L^*$, and hence $L$, is a Lie algebra, a contradiction. Hence $\dim R^*\neq 0$. 
\par

Moreover, $R\subseteq U$. If $R=U$ then $R$ is a maximal solvable subalgebra of $L$, which is impossible unless $R=L$. But then the result follows from Proposition \ref{p:solv}. So suppose $\dim R=0,1$. The former implies that $L$ is a semisimple Lie algebra, which is impossible. The latter implies that $L=S\oplus Fa$, where $S$ is a semisimple Lie algebra. But this is also a Lie algebra and so is impossible. 
\par

It follows that $U^*$ must be a maximal solvable subalgebra of $L^*$. The result now follows from Lemma \ref{l:int}. 
\end{proof}
\medskip

A subalgebra $U$ of $L$ is called {\em lower semi-modular} in $L$ if $U\cap B$ is maximal in $B$ for every subalgebra $B$ of $L$ such that $U$ is maximal in $\langle U,B\rangle$. We say that $L$ is lower semi-modular if every subalgebra of $L$ is lower semi-modular in $L$. 

\begin{coro}  Let $L$ be a non-Lie supersolvable Leibniz algebra over a field of characteristic zero. Then $L^*$ is supersolvable.
\end{coro}
\begin{proof} We have that $L$ is solvable and lower semi-modular, by \cite[Proposition 5.1]{st}. It follows from Proposition \ref{p:solv} that the same is true of $L^*$. Hence $L^*$ is supersolvable, by \cite[Proposition 5.1]{st} again.
\end{proof}

\section{Nilpotent Leibniz algebras}
A Lie algebra $L$ is callled {\em almost nilpotent of index $n$} if it has a basis 
\[ \{x;e_{11},\ldots,e_{1r_1};\ldots;e_{n1},\ldots,e_{nr_n}\}
\] such that
\begin{align*}
-[e_{ij},x]=[x,e_{ij}] & =e_{ij}+e_{i+1,j} \hbox{ for } 1\leq i\leq n-1, 1\leq j \leq r_i,\\
-[e_{nj},x]=[x,e_{nj}] & =e_{nj} \hbox{ and } r_j\leq r_{j+1} \hbox{ for } 1\leq j\leq n-1
\end{align*} all other products being zero.
\medskip

The following result was proved in \cite{alnilp}

\begin{theor} Let $L$ be a nilpotent Lie algebra of index $n$ and of dimension greater than two for which $L^*$ is not nilpotent, over a field of characteristic zero. Then $L^*$ is almost nilpotent of index $n$. Moreover, every almost nilpotent Lie algebra is lattice isomorphic to a nilpotent Lie algebra.
\end{theor}

For non-Lie Leibniz algebras we have the following result.

\begin{theor}\label{t:nilp} Let $L$ be a nilpotent non-Lie Leibniz algebra over a field of characteristic zero. Then $L^*$ is a non-Lie nilpotent Leibniz algebra.
\end{theor}

First we need a lemma.

\begin{lemma}\label {l:nilpmin} Let $L$ be a nilpotent Leibniz algebra and let $W=Fw$ be a minimal ideal of $L$ contained in the Leibniz kernel, $I$, of $L$. Then $W^*$ is a minimal ideal of $L^*$ and $W^*\subseteq Z(L^*)$.
\end{lemma}
\begin{proof}  Suppose that $x\notin I$, where $x^n=0$ but $x^{n-1}\neq 0$. Then $S=\langle x,W\rangle=\langle x\rangle+W$ and $\langle x \rangle\cap W=0$ or $1$. The former implies that $\langle x\rangle$ is a maximal subalgebra of $S$, whence $\langle x^*\rangle$ is a maximal subalgebra, and hence an ideal, of $S^*=\langle x^*, W^*\rangle$. The latter implies that $W\subseteq \langle x\rangle$, whence $W^*\subseteq \langle x^*\rangle$. In either case, $[w^*,x^*]\in \langle x^*\rangle \cap I^*$. Hence $[w^*,x^*]=\sum_{i=2}^n\lambda_i(x^*)^i$.
\par

Suppose that $\lambda_2\not = 0$ and consider $\langle \lambda_2 x-w\rangle$. If $W\subseteq \langle x \rangle$ then $W=Fx^n$ and $\langle \lambda_2 x-w\rangle=\langle x \rangle$. If $W\not \subseteq \langle x\rangle$ then $(\lambda_2 x-w)^k =\lambda_2^kx^k-\lambda_2^{k-1}\mu^{k-1}w$, where $[w,x]=\mu w$. In either case, $\langle \lambda_2 x-w\rangle$ is a cyclic subalgebra of dimension $n$.

However, $$(\lambda_2x^*-w^*)^2=\lambda_2^2(x^*)^2-\lambda_2\sum_{i=2}^n\lambda_i(x^*)^i=\lambda_2\sum_{i=3}^n\lambda_i(x^*)^i,$$ so $\langle \lambda_2x^*-w^*\rangle$ is a cyclic subalgebra of dimension $n-1$, contradicting Corollary \ref{c:cyclic}. It follows that $\lambda_2=0$. A similar argument shows that $\lambda_i=0$ for all $2\leq i\leq n$, so $[w^*,x^*]=0$. Also, $[x^*,w^*]=0$, since $w^*\in I^*$, from which the result follows.
\end{proof}
\medskip

Now we can prove Theorem \ref{t:nilp}.
\medskip

\begin{proof} We have $L/L^2$ is abelian and $L^2=\phi(L)$, so $L^*/\phi(L^*)$ is almost abelian or three-dimensional non-split simple. The latter is impossible, as it would imply that $L^*=\phi(L^*)\dot{+}S^*=S^*$, where $S^*$ is three-dimensional non-split simple, by Theorem \ref{t:levi}. But then $L$ is a two-dimensional Lie algebra, by Theorem \ref{t:lie}, a contradiction. It follows that $L^*/\phi(L^*)$, and hence $L^*$, is supersolvable (see \cite[Theorems 3.9 and 5.2]{barnes2}) and has nilradical $$N^*=\phi(L^*)+Fe_{11}^*+\dots+Fe_{1r_1}^*.$$  Let $L$ be a minimal counter-example, so $L$ is non-Lie and nilpotent, but $L^*$ is not nilpotent. 
\par

 Now $I$ is non-zero, so choose a minimal ideal $W=Fw$ of $L$ inside $I$. We have that $W^*$ is a minimal ideal of $L^*$ inside $Z(L^*)$, by Lemma \ref{l:nilpmin}. Then $L^*/W^*$ is not nilpotent, so $L/W$ is a Lie algebra and $L^*/W^*$ is almost nilpotent. Hence there is a basis
\[ \{x^*;e_{11}^*,\ldots,e_{1r_1}^*;\ldots;e_{n1}^*,\ldots,e_{nr_n}^*, w^*\} \hbox{ for } L^*
\] such that
\begin{align*}
[x^*,e_{ij}^*] & =e_{ij}^*+e_{i+1,j}^*+\lambda_{ij}w^* \hbox{ for } 1\leq i\leq n-1, 1\leq j \leq r_i,\\
[x^*,e_{nj}^*] & =e_{nj}^*+\lambda_{nj}w^* \hbox{ and } r_j\leq r_{j+1} \hbox{ for } 1\leq j\leq n-1, 
\end{align*}  where $\lambda_{ij}\in F$, $I^*=Fw^*$ and $(N^*)^2\subseteq W^*$. Let $M^*$ be spanned by all of the basis vectors for $L^*$ apart from $e_{11}^*$. Then $M^*$ is not nilpotent and has nilradical $F^*$ spanned by all of the basis vectors apart from $e_{11}^*$ and $x^*$. By the minimality, we must have that $M$ is Lie and $M^*$ is almost nilpotent, so $(F^*)^2=0$, $(x^*)^2=0$ and $[e_{ij}^*,x^*]=-[x^*,e_{ij}^*]$ for all of the $e_{ij}^*$'s apart from $e_{11}^*$. But also $[e_{11}^*,x^*] = -e_{11}^*-e_{21}^*+\mu w^*$ for some $\mu \in F$, so
\begin{align*}
[e_{11}^*,x^*] & =[[x^*,e_{11}^*],x^*]-[e_{21}^*,x^*] \\
 & =[x^*,[e_{11}^*,x^*]]+[(x^*)^2,e_{11}]+[x^*,e_{21}^*] \\
 & =-[x^*,e_{11}^*]-[x^*,e_{21}^*]+[x^*,e_{21}^*]=-[x^*,e_{11}^*]
\end{align*}
\par

We now claim that $(N^*)^2=0$. It suffices to show that $[N^*,e_{11}^*]=0$, which we do by a backwards induction argument. We have, for any $f^*\in F^*$,
\begin{align} [f^*,e_{11}^*] & = [f^*,[x^*,e_{11}^*]-e_{21}^*-\lambda_{11}w^*]=[f^*,[x^*,e_{11}^*]] \nonumber \\
 & = [[f^*,x^*],e_{11}^*]-[[f^*,e_{11}^*],x^*]= [[f^*,x^*],e_{11}^*],
\end{align} since $[f^*,e_{11}^*]\in W\subseteq Z(L^*)$. Now putting $f^*=e_{nj}^*$ gives
\[ [e_{nj}^*,e_{11}^*]=[[e_{nj}^*,x^*],e_{11}^*]=-[e_{nj}^*,e_{11}^*],
\] whence $[e_{nj}^*,e_{11}^*]=0$. So now suppose that $[e_{ij}^*,e_{11}^*]=0$ for some $2\leq i\leq n$.

Putting $f^*=e_{i-1,j}^*$ ($(i-1,j)\neq (2,1)$) in (1) gives
\[ [e_{i-1,j}^*,e_{11}^*]=[[e_{i-1,j}^*,x^*].e_{11}^*]=-[e_{i-1,j}^*,e_{11}^*]
\] which, again, yields that $[e_{i-1,j}^*,e_{11}^*]=0$. Finally, note that, if we now put $f^*=e_{11}^*$, then (1) remains valid, so $(e_{11}^*)^2=0$
and $(N^*)^2=0$.
\par

Now replace $e_{nj}^*$ by $e_{nj}^*+\lambda_{nj}w^*$, $e_{ij}^*$ by $e_{ij}^*+(-1)^{n-i}\lambda_{i+1,j}w^*$ to see that $L^*$ is almost nilpotent and $L$ is a Lie algebra, a contradiction. Hence the result holds.
\end{proof}

\end{document}